\documentclass{amsart}

\usepackage[pagebackref=true,colorlinks=true, pdfstartview=FitV, linkcolor=blue,citecolor=red, urlcolor=blue]{hyperref}

\usepackage{amsmath,amsthm,amssymb,mathrsfs}
\usepackage{color}
\usepackage{hyperref}
\usepackage{url}
\usepackage{booktabs}

\newcommand{\bburl}[1]{\textcolor{blue}{\url{#1}}}

\newtheorem{thm}{Theorem}[section]
\newtheorem{cor}[thm]{Corollary}
\newtheorem{lem}[thm]{Lemma}
\newtheorem{prop}[thm]{Proposition}

\theoremstyle{definition}

\theoremstyle{definition}

\theoremstyle{remark}
\newtheorem{rem}[thm]{Remark}

\newcommand\be{\begin{equation}}
\newcommand\ee{\end{equation}}
\newcommand\ben{\begin{enumerate}}
\newcommand\een{\end{enumerate}}
\numberwithin{equation}{section}

\newcommand{\Z}{\ensuremath{{\bf Z}}}

\newcommand{\Q}{{\bf Q}}

\newcommand{\N}{{\bf N}}

\title{Shifts of the prime divisor function of Alladi and Erd\H{o}s }

\author{Snehal Shekatkar}
\email{\textcolor{blue}{\href{snehal.shekatkar@cms.unipune.ac.in}{snehal.shekatkar@cms.unipune.ac.in}}}
\address{Centre for modeling and simulation, S.P. Pune University, Pune, Maharashtra, 411007 India}

\author{Tian An Wong}
\email{\textcolor{blue}{\href{tiananwong@iiserpune.ac.in}{tiananwong@iiserpune.ac.in}}}
\address{Indian Institute for Science Education and Research (IISER) Pune, Dr Homi Bhabha Road, Pune, Maharashtra, 411008 India}



\date{\today}

\begin{document}

\begin{abstract}
We introduce a variation on the prime divisor function $B(n)$ of Alladi and Erd\H{o}s, a close relative of the sum of proper divisors function $s(n)$. After proving some basic properties regarding these functions, we study the dynamics of its iterates and discover behaviour that is reminiscent of aliquot sequences. We prove that no unbounded sequences occur, analogous to the Catalan-Dickson conjecture, and give evidence towards the analogue of the Erd\H{o}s-Granville-Pomerance-Spiro conjecture on the pre-image of $s(n)$. 
\end{abstract}

\maketitle



\section{Introduction}


Let $n$ be a positive integer with prime factorization $n=p_1^{r_1}\dots p_k^{r_k}$. Consider the sum of prime divisors function $B(n) = \sum_{i=1}^kr_ip_i$ and the sum of distinct prime divisors function $\beta(n) = \sum_{i=1}^k p_i$. They can be viewed as variants of the sum of proper divisors function $s(n)=\sum_{d||n}d$, discussed since Pythagoras. The arithmetic properties of $B(n)$, for example, were studied by Alladi and Erd\"os \cite{AE}, and $\beta(n)$ studied by Hall \cite{H} since the 1970s.\footnote{In \cite{AE}, the authors denote $B(n),\beta(n)$ by $A(n), A^*(n)$. In this paper we follow the former convention, consistent with more recent literature.} These are large additive functions in the sense that they have the same average order as the largest prime factor of $n$, which is expected by the well-known result of Hardy and Ramanujan that $\Omega(n)$ and $\omega(n)$ have the same average order $\log\log n$.

In this paper, we introduce perturbations of $\beta(n)$ and $B(n)$, by shifting the values at certain fixed points. Clearly $B(n)=\beta(n)=n$ if $n$ is prime. Then define for a fixed positive integer $n$, the functions
\be
B_a(n) = 
\begin{cases}
n+a		& n \text{ prime}\\
B(n)		& n \text{ otherwise}
\end{cases}
\ee
and similarly $\beta_a(n)$. They are no longer additive, and behave in similar and different ways in comparison to the original functions, as we discuss in Section \ref{shifts}. 

Our main interest will lie in their iterates, denoting $B^2_a(n)= B_a(B_a(n))$, and similarly for $\beta_a(n)$. The aliquot sequence $n, s(n), s(s(n)),\dots$ stops at either primes, or cycles of length two (amicable pairs) or longer (sociable numbers). Iterating $B_a(n)$ and $\beta_a(n)$, we encounter similar phenomena, detailed in Section \ref{iterates}. Surprisingly, for fixed $a$ only a handful of cycles are observed; it is in varying $a$ we find many different cycles. 

This suggests a version of the Catalan-Dickson conjecture, which states that there does not exist unbounded aliquot sequences (or, alternatively, the Guy-Selfridge counter-conjecture \cite{GS} which gives certain candidate counterexamples). It is obvious that $B_0^k(n)$ and $\beta_0^k(n)$ all eventually reach a fixed point, but for general $a$ there may be exceptional cases in which the sequences escape to infinity. In support of this, the Green-Tao theorem guarantees that there exists integers $n$ and $a$ such that
\be
n<B_a(n) < \dots < B_a^k(n), \text{ and }\ n<\beta_a(n) < \dots < \beta_a^k(n).
\ee
that for any $k>0$ (Corollary \ref{GT}). Nonetheless, we show in Theorem \ref{finite}, that in fact no unbounded sequences occur.
\begin{thm}
There exists finitely many cycles for each $a$, and all integers iterate to cycles. 
\end{thm}

We conclude with evidence towards a variant of a conjecture of Erd\H{o}s, Granville, Pomerance, and Spiro \cite{EGPS}, which states that the pre-image of any set of asymptotic density zero under $s(n)$ is again of asymptotic density zero. Namely, we prove as Theorem \ref{EGPS}, using previous work of Pollack, Pomerance, and Thompson \cite{PPT} on divisor-sum fibres:

\begin{thm}
Let $\epsilon=\epsilon(x)$ be a fixed function tending to 0 as $x\to\infty$, and let $\mathscr A$ be a set of at most $x^{\frac12+\epsilon(x)}$ positive integers. Then
\be
\#\{n\leq x : B(n)\in\mathscr A\} = o_\epsilon(x)
\ee
as $x\to\infty$, uniformly in the choice of $\mathscr A$.
\end{thm}

\subsubsection*{Acknowledgments}  The authors would like to thank Jeff Lagarias for suggesting the proof of Theorem \ref{finite}.

\section{Shifts}
\label{shifts}

In this section, we prove some basic results of $B_a(n)$ and $\beta_a(n)$. First, from the fact that $B(n^k)=kn$ it follows that if $n$ is composite then $B_a(n^k)=kn$. Also, $B_a(n)$ is close to additive in the sense that $B_a(mn)=B_a(n)+B_a(m)$ if and only if $m,n$ are not prime.


The following result shows that $B_a(n)$ is in certain ways similar to $B_0(n)$, independently of $a$.

\begin{prop}
\label{order}
(a) The average order $B_a(n)$ is $\dfrac{\pi^2 n}{6\log n}$. In other words,
\be
\sum_{n\leq x} B_a(n) \sim \frac{\pi^2 x^2}{12\log x}.
\ee
(b) The average order of $B_a(n) - \beta_a(n)$ is $\log\log n$. In other words,
\be
\sum_{n\leq x}(B_a(n) - \beta_a(n)) = x \log\log x + O(x).
\ee
(c) For any fixed integer $N$, the set $\#|B_a(n) - \beta_a(n) = N|$ has positive natural density. In other words,
\be
\lim_{n\to\infty} \frac{n}{B_a(n) - \beta_a(n)} > 0.
\ee
\end{prop}

\begin{proof}
These follow simply from the fact that
\be
\sum_{n\leq x} B_a(n) = \sum_{n\leq x} B(n) + a \sum_{p\leq x} 1
\ee
and similarly for $\beta_a(n)$, then applying the corollaries to \cite[Theorem 1.1]{AE}. The additional sum over $p$ is $a\pi(x)$, which gives a smaller error $ax/\log x$. For the last two claims, we also note that $B_a(n) - \beta_a(n)=B(n) - \beta(n).$
\end{proof}

Let $f(n)$ be an arithmetic function taking positive integer values. Then for any fixed $N\geq 0$, the local density of $f(n)$ is defined to be
\be
d_N = \lim_{x\to \infty} \frac{1}{x}\#|n\leq x : f(n) = N|.
\ee
Then we may deduce the following result from \cite[\S3]{Iv}:

\begin{cor}
There exists $d_N$ such that
\be
\#\{n\leq x : B_a(n)-\beta_a(n) = N\} = d_N x + O(x^\frac12\log x)
\ee
and
\be
\#\{n\leq x : B_a(n)-\beta_a(n) \geq N\} \ll \frac{x}{N}.
\ee
uniformly in $N>0$.
\end{cor}


The next theorem shows a departure from the unshifted case. In fact, we find a qualitative difference depending on whether $a$ is 0, even, or odd. Recall that in the case of $a=0$, we have that $B_0(n)$ and $\beta_a(n)$ are additive, and therefore uniformly distributed for all $q$ by Delange's theorem (c.f. \cite{G}). 

\begin{thm}
\label{dist}
(a) Let $q=2$. Then $B_a(n)$ is uniformly distributed \textnormal{mod} $q$ if $a$ is even. That is, there is a constant $c_0>0$ such that
\be
\sum_{n=1}^\infty (-1)^{B_a(n)} = O(x\exp(c_0(\log x)^\frac23).
\ee
If $a$ is odd, we have instead $O(x)$, as $x\to\infty$.

(b) Let $q,h$ be fixed integers with $q>2$. Then for any $a$ such that $q|ha$,  $B_a(n)$ is uniformly distributed mod $q$.
\end{thm}

\begin{proof}
First consider $q=2$. Notice that $(-1)^{B_a(p)}= (-1)^{B(p)}(-1)^a$, so the even case follows directly from \cite[Theorem 3.1]{AE}, so it suffices to consider the case with $a$ odd. We have the Dirichlet series
\be
\sum_{n=1}^\infty\frac{(-1)^{B_a(n)}}{n^s} = \sum_{n=1}^\infty\frac{(-1)^{B(n)}}{n^s} - 2\sum_{p} \frac{1}{p^s}
= \frac{2^s+1}{2^s-1}\frac{\zeta(2s)}{\zeta(s)} - 2P(s)
\ee
where $P(s)=\sum_{p} p^{-s}$ denotes the so-called prime zeta function, which is essentially $\log\zeta(s)$ as $s\to1^+$.

The claim for $q>2$ follows from a recent result of Goldfeld: by assumption the character $\exp(2\pi i \frac{hB_a(n)}{q}) = \exp(2\pi i \frac{hB(n)}{q})$ is completely multiplicative, and thus we may apply \cite[Theorem 1.2]{G} to get
\be
\sum_{\substack{n\leq x\\ B(n) \equiv h\textnormal{ mod } q}}1 = \frac{x}{q} + O\left(\frac{x}{(\log x)^\frac12}\right)
\ee
as $x\to\infty$, which gives the result.
\end{proof}

From the above we observe that the lack of additivity leads to obstructions to uniform distribution, which in turn leads to a departure from certain proofs related to iterates of $B_a(n)$ and $\beta_a(n)$, which we will discuss next. Nonetheless, we can conclude:

\begin{cor}
If $a$ is even, then 
\be
\sum_{n=1}^\infty\frac{(-1)^{B_a(n)}}{n}=0.
\ee
\end{cor}


\section{Iterates}
\label{iterates}

Now we turn to the dynamics of iterates of $B_a(n)$ and $\beta_a(n)$. While most of the results below concern $B_a(n)$, many of them can be carried over to $\beta_a(n)$ without difficulty, which we leave to the interested reader. Throughout we will also raise several problems regarding the sequences $B^k_a(n)$, which can also be posed for $\beta^k_a(n)$, suggested by the numerical computations.

\subsection{Cycles}
\label{cycles}

We call $n$ a periodic point if $B_a^{k+l}(n)=B_a^k(n)$ for some $l$, and eventually periodic if $B_a^k(n)$ is periodic for some $k$. Define a cycle to be the orbit of a periodic point $n$. We will sometimes refer to the fixed points $B_a(n)=n$ as trivial cycles. 

\begin{lem}
For any $a$ nonzero, $B_a(n) = n$ if and only if $n=4$.
\end{lem}

\begin{proof}
Primes are not fixed points by definition. So the only possible fixed points are those of $B(n)$ for $n$ composite, which is $n=4$, as $B(n)<n$ for any composite $n>4$.
\end{proof}

\subsubsection{The first case} 

Recall that the stopping time for $n$, denoted $\sigma(n)$ is defined to be $\inf\{k: B^k_a(n)<n\}$. We will also define the total stopping time for $n$, denoted $\sigma_\infty(n)$ to be the number of iterates required for $B^k_a(n)$ to enter a cycle.

\begin{prop}
$B_1^k(n)$ is eventually periodic for all $n$, with cycles $(4)$ and $(5,6)$. 
\end{prop}

\begin{proof}
This is easy to check the orbits for small $n$, say $n\leq 6$, so we may assume that $n>6$. We first claim that $B^k_1(n)$ has stopping time $\sigma(n)=2$ if $n$ is prime, otherwise $\sigma(n)=1$. It suffices to check for $n=p$. Set $p+1=2m$. Then 
\be
B^2_1(p)=B_1(p+1)=2 + B_1(m)\leq 2+m < p. 
\ee
It follows from the claim that $B^k_1(n) < n$ for all $k>1$, thus $B^{2k}_1(n)$ is strictly decreasing until it reaches the cycle $(5,6)$.
\end{proof}

\begin{rem}
The proof above fails for $a>1$, since $p+a$ may be prime if $a$ is even, while the final inequality $2+m<p$ is no longer guaranteed if $a$ is odd. 
\end{rem}

\subsubsection{Amicable pairs} 
Do all primes above 3 occur in some amicable pair, i.e., 2-cycle for some $a$? We can answer this in the affirmative.

\begin{prop}
\label{amicable}
Every $p>3$ occurs in a $2$-cycle $(p,B_a(p))$ for some integer $a$.
\end{prop}
\begin{proof}
We have to produce an $a$ such that $B_a(p+a)=p$. This forces $p+a$ to be composite, so it suffices to ask if $B_0(p+a)=p$ for some $a$. Let $n$ be a composite solution to $B_0(n)=p$. Then setting $a=n-p$, we have that $B_a: p \mapsto n \mapsto p$ as desired.

It remains to show that $B_0(n)=p$ always has a composite solution. Let $q$ be the largest prime less than $p$. If $p-q$ is prime, then choose $n=q(p-q)$. If not, then let $a$ be a prime dividing $p-q$, and write $p-q=ab$. Then choose $n=qa^b$.
\end{proof}

\begin{rem}
Observe that if $a$ was chosen to be the largest prime dividing $p-q$, then the construction will in fact produce the smallest composite solution $n$. 
\end{rem}

\subsubsection{Finite cycles}
How many kinds of cycles appear for fixed $a$? Up to $a\leq 200$, we found at most 4 distinct nontrivial cycles, at $a=39$, we have:
\[
(43, 82),(13,52,17,56),(7,46,25,10),(5, 44, 15, 8, 6)
\]
Note that from the above, we see that different $k$-cycles can occur for a fixed $k$ and $a$. 

It is natural to consider a variant of the Catalan-Dickson conjecture for $s(n)$: Does the sequence $B^k_a(n)$ ever escape to infinity as $k\to\infty$? Heuristically, the sequence $B_a(n)/n$ has average order $\pi^2/6\log n < 1$ by Proposition \ref{order}(a),\footnote{In contrast to $s(n)/n$ which is slightly greater than 1 on average.} which suggests that an unbounded sequence should not exist. We show that this is indeed true:

\begin{thm}
\label{finite}
There exists finitely many cycles for each $a$, and all integers iterate to cycles. 
\end{thm}

\begin{proof}
It suffices to show that for any shift $a$ that there is a bound $C(a)$ such that all orbits of iterating $B_a$ from any starting point $n$ enter the range $[1, C(a)]$. Without loss of generality take $n$ be a prime $p$. We will show that for $p$ large enough, the next iterate of $p$ which goes downhill will go to a number less than $p$.

Let $s$ be the smallest prime that does not divide $a$, which is less than $2a$ by Bertrand's postulate. Then $a$ mod $s$  is nonzero, hence relatively prime to it. On the other hand, at least one of $p+a,p+2a,\dots, p+sa$ is divisible by $s$, and therefore composite. Denote it by $p+ka$. 

We claim that for composite $n$, $B_a(n) \le 2 + \frac{n}{2}$. Assuming this, then whenever $p$ is such that
\be
B_a(p+ka)\le 2+ \frac{p+ sa}{2}< p
 \ee
then we will go downhill.  Now since $s < 2p$, we may choose $C(a)=2a^2 + 10 < p.$

It remains to prove the claim. Since $n$ is composite, it suffices to prove it for $a=0$. Let $n = p^kq$ where $p$ is the smallest prime factor of $n$, with $(p,q)=1, q>1$. (If $q=1$ then $B(p^k) = kp$ and we are done.) By additivity, we have then:
\begin{align}
B(n)  &= kp + B(q) \\
& \le kp + q \\
& \le 2 + (\frac{p^k}{2}-1)q + q = 2 + \frac{n}{2}.
\end{align}
The last inequality follows since $2 + (\frac{p^k}{2}-1)q \ge p^k$, and $q >1$. 
\end{proof}

The following table lists the distinct nontrivial cycles found for small $a$ and checking $n$ up to $10^6$.
\begin{table}[ht!]
  \centering
  \caption{Nontrivial cycles for $n\leq 10^6$.}
  \label{tab:table1}
  \begin{tabular}{c| l | c| l}
    \toprule
    $a$ & cycles & $a$& cycles \\
    \midrule
    1 & $(5,6)$				&11&$(5,15,8,6)$\\
    2 & $(5,7,9,6)$			&12&$(5, 17, 29, 41, 53, 65, 18, 8, 6)$\\
    3 & $(5,8,6),(7,10)$		&13&$(5,16,8,6)$\\
    4 & $(5,9,6)$				&14&$(5, 19, 33, 14, 9, 6), (7, 21, 10)$ \\
    5 & $(7,12)$				&15&$(5,20,9,6),(19,34)$\\
    6 & $(7,13,19,25,10)$		&16&$(7, 23, 39, 16, 8, 6, 5, 21, 10)$\\
    7 & $(5,12,7,14,9,6)$		&17&$(7, 24, 9, 6, 5, 22, 13, 30, 10),(11,28)$\\
    8 & $(5,13,21,10,7,15,8,6)$	&18&$(5, 23, 41, 59, 77, 18, 8, 6),(7, 25, 10)$ \\
    9 & $(5,15,9,6),(13,22)$	&19&$(5, 24, 9, 6)$ \\
    10 & $(5,15,8,6)$			&20&$(5, 25, 10, 7, 27, 9, 6)$\\
    \bottomrule
  \end{tabular}
\end{table}

\subsubsection{Cycle length} 
What are the lengths of cycles, and how do they depend on $a$ and $n$? What are the stopping times $\sigma(n)$ and $\sigma_\infty(n)$? We have not studied this question in detail, but numerical experiments suggest that both times are small, relative to $s(n)$, for example.

\subsubsection{Sign patterns}
Any $k$-cycle $(n,B_a(n),\dots,B^k_a(n))$ can be ordered so that $n$ is the least term in the sequence, making $n$ prime and $B^k_a(n)$ composite. We adopt the following notation: we will assign $\{+,-\}$ to denote in a cycle whether a number is prime or composite. For example, the cycle $(5,7,9,6)$ in $a=2$ has sign pattern $(+,+,-,-)$.

We can now pose the following question: what are the possible sign patterns allowed in a $k$-cycle? All non-trivial cycles of length $k>2$ must have sign pattern of the form $(+,\dots,-)$. Do all combinations occur in between? For example, with $k=3$ we find both combinations $(+,+,-)$ and $(+,-,-)$ to occur. 


\subsection{Ascending chains}

Looking in the other direction, a number $n$ is called abundant if the sum of proper divisors $s(n)>n$, and it was shown in an unpublished work of Lenstra, and later improved by Erd\H{os} \cite{E}, that for every $K$ there is an $n$ for which
\be
\label{abundantchain}
n<s(n) < \dots < s^K(n).
\ee
Our $B_a(n)$ are simpler in the sense that $B_a(n)>n$ if and only if $n$ is prime and $a>1$, but on the other hand the analog of Lenstra's result requires the Green-Tao theorem \cite[Theorem 1.1]{GT}, which implies that there are arbitrarily long arithmetic progressions of primes, since such a progression yields primes $B_a^k(p) = p+ka$. Nonetheless we can immediately conclude:

\begin{cor}
\label{GT}
For any $k>0$, there exists integers $n$ and $a$ such that
\be
n<B_a(n) < \dots < B_a^k(n), \text{ and }\ n<\beta_a(n) < \dots < \beta_a^k(n).
\ee
\end{cor}

\begin{rem}
Note that there is a natural extension of $B_a(n)$ to $\Z_{\geq0}$ by setting $B(1)=1$ and $B(0)=0$, and to negative integers by setting $B(-n)=B(n)$, similarly for $\beta_a(n)$. Thus iterating our functions can be viewed as studying dynamics on $\Z$ itself. We may also extend to $\Q$ by defining $B(\frac{x}{y}):=B(x)-B(y)$ for reduced fractions, though upon iterating once we return to $\Z$, and from there $\Z_{\geq0}$. One possible way of producing more interesting extensions is by setting $B(-n) = -B(n)$, and $B(\frac{x}{y})=B(x)/B(y).$
\end{rem}

\subsection{Prime-divisor fibres}
A question related to cycles is: For a fixed $p$, what is the set of solutions $\#\{n : B_a(n)=p\}$? The solution sets are the same for every $a$, except possibly the pre-image $\{p-a\}$, which will be counted if it is prime. From the proof of Theorem \ref{amicable} we have already found the smallest composite solution. 

In fact, the solutions to $B_a(n) = m$ for a fixed $m$ are given precisely by the prime partitions $\kappa(m)$ of $m$, as was already observed in \cite[Theoren 2.7]{J}, and there is at least one composite solution for all $n\ge 5$. From this fact we can immediately deduce from \cite[VIII.26]{FS} the following asymptotic:

\begin{prop}
We have
\be
\log(\#\{n : B_a(n)=m\}) \sim 2\pi\sqrt{\frac{m}{3\log m}}.
\ee
as $m\to\infty$.
\end{prop}

\begin{rem}
Indeed, one even has a recursive definition for $\kappa(n)$ in terms of $\beta(n)$,
\be
\kappa(n)=\frac{1}{n}\big(\beta(n)+\sum_{i=1}^{n-1}\kappa(n-i)\beta(i)\big).
\ee 
with the initial condition $\kappa(1)=0$. In other words, the number solutions of $B(n)=m$ are determined by the values of $\beta(i)$ for $i\leq m$. 
\end{rem}

More generally, this can be phrased in terms of prime-divisor sum fibres, with reference to \cite{PPT}: Let $\mathscr{A}\subset \N$ be a set of asymptotic density zero (for example, the set of prime numbers). What is the preimage $B^{-1}_a(p)$ for a fixed $a,p$? Erd\H{o}s, Granville, Pomerance, and Spiro (EGPS) conjecture that the fibre $s^{-1}(\mathscr{A})$ also has asymptotic density zero. The following theorem shows that evidence towards the analogous conjecture also holds for $B(n) = B_0(n)$.

\begin{thm}
\label{EGPS}
Let $\epsilon=\epsilon(x)$ be a fixed function tending to 0 as $x\to\infty$, and let $\mathscr A$ be a set of at most $x^{\frac12+\epsilon(x)}$ positive integers. Then
\be
\#\{n\leq x : B(n)\in\mathscr A\} = o_\epsilon(x)
\ee
as $x\to\infty$, uniformly in the choice of $\mathscr A$.
\end{thm}

\begin{proof}
Let $a\in\mathscr A$, and $n$ a non-exceptional preimage of $a$ in the sense of \cite[\S2]{PPT}. Write $n=de$ where $d$ is the largest divisor of $n$ not exceeding $\sqrt{x}$. Then following the proof of \cite[Theorem 1.2]{PPT}, we observe firstly that from $B(n)\leq s(n)$ that
\be
B(n) \ll x^{\frac12-10\epsilon(x)}\log x. 
\ee 
Secondly, $B(de) = B(d) + B(e)$ let $g=\gcd(B(d),B(e))$, then 
\be
\frac{a}{g} \equiv  \frac{B(e)}{g} \pmod{B(d)/g},
\ee
so that given $d$ we see that $B(e)$ lies in a uniquely determined residue class mod $B(d)/g$. Hence the number of choices of $B(e)$ given $d$ is
\be
\ll 1+x^{\frac12-10\epsilon(x)}\log x\frac{g}{B(d)}\le  1+x^{\frac12-10\epsilon(x)}\frac{\log x}{\log d}
\ee
where the second inequality follows from $B(d) \ge \log d $.

Now we sum over possible values of $g$ and $d$. We can write each $d$ as $gh$ where $h\leq x^{1/2-10\epsilon(x)}/g$. Thinking of $g$ as fixed and summing on $d = gh$ gives a bound of $\ll (\log x)^2x^{1/2-10\epsilon(x)}$. Whereas summing over the $\tau(a)$ divisors $g$ of $a$ bounds the number of possibilities for $n$ by 
\be
\ll \tau(a)(\log x)^2x^{1/2-10\epsilon(x)}<x^{1/\log\log x}x^{1/2-10\epsilon(x)}\leq x^{1/2-9\epsilon(x)},
\ee
bounding the number of $n$ that arise in this way as desired.
\end{proof}

\begin{rem}
Since $B_a(n)\leq s(n)$ is no longer true for $a>1$, it remains to ask whether the above holds for general $a$, and if moreover the EGPS conjecture should also hold. \end{rem}



\bibliography{AlladiErdos}
\bibliographystyle{alpha}

\end{document}